\documentclass[12pt, twoside]{article}

\usepackage{amsthm}
\usepackage{enumerate}

\usepackage{booktabs}
\usepackage{array}

\usepackage{amstext}
\usepackage[all]{xy}

\usepackage[toc,page]{appendix}

\usepackage{fancyhdr}
\usepackage{lastpage}
\usepackage{hyperref}
\usepackage{bbding}
\usepackage{geometry}

\usepackage{amscd}

\usepackage{mathrsfs}
\usepackage{amsfonts}
\usepackage{amssymb}
\usepackage{pifont}
\usepackage{amsmath}
\usepackage{mathrsfs}
\usepackage{enumerate}
\usepackage{amsmath}
\usepackage{mathrsfs}
\usepackage{amsfonts}
\usepackage{amsmath}
\usepackage{amscd}
\usepackage[all]{xy}

\usepackage{color}
\definecolor{DarkBlue}{rgb}{0.1,0.1,0.5}
\definecolor{Red}{rgb}{0.9,0.0,0.1}
\definecolor{Navy}{rgb}{0.00,0.00,0.30}
\definecolor{Yellow}{rgb}{1.00,1.00,0.00}
\definecolor{Gold}{rgb}{1.00,0.84,0.00}
\definecolor{Lightgoldenrod}{rgb}{0.93,0.87,0.51}
\definecolor{Goldenrod}{rgb}{0.85,0.65,0.13}
\definecolor{Black2}{rgb}{0.00,0.00,0.00}
\definecolor{orange}{rgb}{0.85,0.65,0.13}
\definecolor{SkyBlue}{rgb}{0.941176,0.972549,1.}
\definecolor{MyLightMagenta}{cmyk}{0.1,0.8,0,0.1}
\definecolor{MistyRose}{RGB}{255,228,225}
\definecolor{CornSilk}{RGB}{255,248,220}

\usepackage{tikz}
\usetikzlibrary{%
  calc,
  fit,
  shapes,
  backgrounds
}

\usepackage{verbatim}

\def\titlerunning#1{\gdef\titrun{#1}}
\makeatletter
\def\author#1{\gdef\autrun{\def\and{\unskip, }#1}\gdef\@author{#1}}
\def\address#1{{\def\and{\\\hspace*{18pt}}\renewcommand{\thefootnote}{}%
\footnote {#1}}%
\markboth{\autrun}{\titrun}}
\makeatother
\def\email#1{e-mail: #1}
\def\subjclass#1{{\renewcommand{\thefootnote}{}%
\footnote{\emph{Mathematics Subject Classification (2010):} #1}}}
\def\keywords#1{\par\medskip
\noindent\textbf{Keywords.} #1}

\numberwithin{equation}{section}

\newcommand{\HH}{\mathrm{H}}
\newcommand{\pp}{\mathbf P^3}
\newcommand{\ch}{\mathrm{ch}}
\newcommand{\nab}{\nu_{\alpha,\beta}}

\newcommand{\db}{\overline {\Delta}}
\newcommand{\dt}{\widetilde{\Delta}}
\newcommand{\bp}{\mathrm{Bl}_{\mathrm{pt}} \mathbf{P}^3}
\newcommand{\Hom}{\mathrm{Hom}}
\newcommand{\ext}{\mathrm{ext}}

\newcommand{\ob}{\overline{\beta}}

\newcommand{\tv}{\widetilde{v}}
\newcommand{\czl}{\mathrm{ch}_0^{\overline{\beta}}(E)}

\newcommand{\czy}{\mathrm{ch}_1^{\overline{\beta}}(E)}

\newcommand{\cze}{\mathrm{ch}_2^{\overline{\beta}}(E)}

\newcommand{\czs}{\mathrm{ch}_3^{\overline{\beta}}(E)}

\newcommand{\czi}{\mathrm{ch}_i^{\overline{\beta}}(E)}

\newcommand{\chy}{\frac{H^2\mathrm{ch}_1}{H^3\mathrm{ch}_0}}

\newcommand{\che}{\frac{H\,\mathrm{ch}_2}{H^3\mathrm{ch}_0}}

\newcommand{\cccp}{\left\{1,\frac{H^2\mathrm{ch}_1}{H^3\mathrm{ch}_0},\frac{H\,\mathrm{ch}_2}{H^3\mathrm{ch}_0}\right\}}

\renewcommand\appendix{\par
  \setcounter{section}{0}
  \setcounter{subsection}{0}
  \setcounter{figure}{0}
  \setcounter{table}{0}
  \renewcommand\thesection{Appendix \Alph{section}}
  \renewcommand\thefigure{\Alph{section}\arabic{figure}}
  \renewcommand\thetable{\Alph{section}\arabic{table}}
}

\newtheorem{theorem}{Theorem}[section]
\newtheorem{defn}[theorem]{Definition}
\newtheorem{prop}[theorem]{Proposition}
\newtheorem{cor}[theorem]{Corollary}
\newtheorem{lemma}[theorem]{Lemma}
\newtheorem{rem}[theorem]{Remark}

\title{\textbf{Stability conditions on Fano threefolds of Picard number one}}
\author{Chunyi Li\\ University of Edinburgh}
\date{\today}

\begin{document}

\baselineskip=17pt


\titlerunning{The University of Edinburgh}

\title{Stability conditions on Fano threefolds of Picard number one}

\author{Chunyi Li}

\date{\today}

\maketitle

\address{School of Mathematics and Maxwell Institute, University of Edinburgh \and
\email{Chunyi.Li@ed.ac.uk}}
\subjclass{Primary 14F05; Secondary 14J45}

\begin{abstract}
We prove the conjectural Bogomolov-Gieseker type inequality for tilt slope stable objects on each Fano threefold $X$ of Picard number one. Based on the previous works \cite{BMT14}, \cite{BMS} and \cite{BMT} on Bridgeland stability conditions, this induces an open subset of geometric stability conditions on D$^b(X)$. We also get a new stronger bound for  Chern characters of slope semistable sheaves on $X$.
\keywords{stability condition, Fano threefolds, Bogomolov-Gieseker type inequality}
\end{abstract}


\section*{Introduction}
The notion of \emph{stability conditions} on a $\mathbb C$-linear triangulated category was introduced by Bridgeland in \cite{Bri07}. The existence of stability conditions on three-dimensional complex varieties is considered to be one of the central open problem in the study of Bridgeland stability conditions. To construct family of geometric stability conditions on a smooth $3$-fold, the general approach is to construct a new heart by tilting the original heart Coh($X$) twice, and then impose suitable central charge function. Following this approach, the main technical difficulty, as revealed by  the work \cite{BMT14} and \cite{BMT}, is to prove some conjectural Bogomolov-Gieseker type inequality that involves the third Chern character $\ch_3$. In special cases when the variety admits a complete exceptional collection, the conjectural Bogomolov-Gieseker type inequality has been proved: the \textbf P$^3$ case is set up in \cite{MacP3} by Macr\`{\i}, and the quadric $3$-fold case is solved in \cite{Sch} by Schmidt. The abelian $3$-fold of Picard number one case has been done in \cite{MP1} and \cite{MP2} by Maciocia and Piyaratne.

In the recent paper \cite{BMS} by Bayer, Macr\`{\i} and Stellari, the authors prove the existence of a stability condition for some $3$-folds including all abelian $3$-folds and Calabi-Yau $3$-folds obtained as a finite quotient of an abelian $3$-folds. In particular, as an important technical result, the authors show that when the polarization $\omega$ and the B-field $B$ are proportional to each other, or in particular, when the Picard number of the  variety is $1$, the general version conjectural Bogomolov-Gieseker type inequality, Conjecture 2.4 in \cite{BMS}, is equivalent to the small limit version Conjecture 5.3. In this paper, we prove the Conjecture 5.3 in \cite{BMS} for smooth Fano $3$-fold of Picard number one.\\

\textbf{Notations}: Let $X$ be a smooth Fano $3$-fold of Picard number one. The Picard group, Pic($X$), is generated by an ample divisor $H$. Let $\alpha>0$ and $\beta$ be two real numbers. We write ch$^{\beta H}_i(E)$ for the $\beta H$-twisted $i$-th Chern character ch$_i(E(-\beta H))$. Adopting the notation in \cite{BMS}, we write
\[\mathrm{Coh}_{\beta}(X):=\langle \mathcal T_{\beta}(X), \mathcal F_{\beta}(X)[1]\rangle,\]
 where $(\mathcal T_{\beta}(X), \mathcal F_{\beta}(X))$ is the torsion pair in $\mathrm{Coh}(X)$ given by:
\begin{center}
$\mathcal T_{\beta}(X):=\langle$ torsion sheaves; torsion-free slope stable sheaves $F$ with $\ch^{\beta H}_1(F)>0\rangle$;\\
$\mathcal F_{\beta}(X):=\langle$ torsion-free slope stable sheaves $F$ with $\ch^{\beta H}_1(F)\leq 0\rangle$.
\end{center}

The \textit{reduced central charge} $\overline {Z}_{\alpha,\beta}$ is defined as follows:

\[\overline {Z}_{\alpha,\beta} (E)= 3\alpha H^2 \ch^{\beta H}_1(E)+i\left(\sqrt{3}H\; \ch^{\beta H}_2(E)-\frac{\sqrt{3}}{2}\alpha^2 H^3 \ch^{\beta H}_0(E)\right).\]
When $\Re\overline{Z}_{\alpha,\beta} (E)\neq 0$,
the \emph{tilt slope function} $\nu_{\alpha,\beta}$ is defined to be the slope $\frac{\Im \overline{Z}_{\alpha,\beta}}{\Re \overline{Z}_{\alpha,\beta}}$ of the reduced central charge. Otherwise,  $\nu_{\alpha,\beta}$ is defined to be $+\infty$.
An object $E$ in Coh$_{\beta}(X)$ is $\nu_{\alpha,\beta}$ \emph{tilt-(semi)stable} if for all non-trial subobject $F\hookrightarrow E$ in Coh$_{\beta}(X)$, one has $\nab(F)<(\leq)\nab(E/F)$.

The real number $\ob(E)$ is defined as

\begin{equation*}
        \ob(E) = \begin{cases}
                        \frac{H\; \ch_2(E)}{H^2\ch_1(E)}  & \text{ if $\ch_0(E) = 0$,} \\
                        \frac{H^2\ch_1(E)-\sqrt{\db_H(E)}}{H^3\ch_0(E)} & \text{ if $\ch_0(E)\neq 0$.}
                    \end{cases}
\end{equation*}
We call an object $E\in D^b(X)$ to be \emph{$\ob$-stable} if there exists an open neighborhood $U\subset \mathbb R^2$ of $(0,\ob(E))$ such that for any $(\alpha,\beta)\in U$ with $\alpha>0$, either $E$ or $E[1]$ is a $\nab$ tilt-stable object of Coh$_\beta(E)$.
\begin{theorem}[Conjecture 5.3 in \cite{BMS}]
Let $X$ be a smooth Fano $3$-fold of Picard number $1$, and $E\in D^b(X)$ be a $\ob$-stable object, then
\[\ch_3^{\ob(E) H}(E)\leq 0.\]
\end{theorem}
By abuse of notations, we will write $\czi$ instead of $\ch_i^{\ob(E) H}(E)$ for short.

Combined with the results in \cite{BMT14} and \cite{BMS}, we get an open subset of \emph{geometric stability conditions} on Stab$(X)$. To make this precise, the second torsion pair ($\mathcal T_{\alpha,\beta}$, $\mathcal F_{\alpha,\beta}$) is defined as follows:
\begin{center}
$\mathcal T_{\alpha,\beta}(X):=\langle$ $\nab$ tilt-slope stable objects $F\in\mathrm{Coh}_\beta(X)$ with $\nab(F)> 0 \rangle$;\\
$\mathcal F_{\alpha,\beta}(X):=\langle$ $\nab$ tilt-slope stable objects $F\in\mathrm{Coh}_\beta(X)$ with $\nab(F)\leq 0 \rangle$.
\end{center}
The heart $\mathcal A_{\alpha,\beta}:=\langle \mathcal T_{\alpha,\beta}(X),\mathcal F_{\alpha,\beta}[1]\rangle$ and the central charge $Z^{a,b}_{\alpha,\beta}$ is
\[\left( -\ch^\beta_3+bH\ch^\beta_2+aH^2\ch^\beta_1\right)+i\left(H\ch_2^\beta-\frac{\alpha^2}{2}H^3\ch^\beta_0\right).\]

Here the coefficients $a$ and $b$ satisfy the inequality $a>\frac{\alpha^2}{6}+\frac{1}{2}|b|\alpha$.
\begin{cor}
The pair $(Z^{a,b}_{\alpha,\beta}, \mathcal A_{\alpha,\beta})$ is a geometric stability condition on $\mathrm D^b(X)$: skyscraper sheaves of points on $X$ are stable with the same phase.
\end{cor}

\textbf{The idea of the proof:} The idea of the proof is based on two main observations. The first is to visualize the tilt slope function via the kernel of its reduced central charge. In this way, we may compare the $\nab$-slope of tilt-stable objects by the `positions' of their Chern characters and the kernel of $\overline{Z}_{\alpha,\beta}$. Based on this observation, when $0<\ob(E)<1$, we have Hom$(\mathcal O(mH),E)=0$ for $m>0$ and Hom$(E, \mathcal O(mH)[1])=0$ for $m\leq 0$. Together with the Serre duality, we get inequalities such as $\chi(\mathcal O(H),E)\leq 0$.

The second observation is that the Todd classes of a Fano variety are `positive' in some sense. Suppose $0< \ob(E)<1$, then when we compute $\chi(O(H),E)$ or $\chi(O(2H),E)$ in terms of $H^{3-i}\czi$, the coefficients of $H^2\czy$ and $H^3\czl$ are expected to be non-negative. Together with the first observation, we get a bound for $\czs$.

This idea works almost directly for Fano $3$-folds with index greater than or equal to two. In the index one case, the inequality $\chi(\mathcal O(H),E)\leq 0$ is not enough to give the claim, because the rest term in the expansion of $\chi(\mathcal O(H),E)$, which involves  $H^2\czy$ and $H^3\czl$, may be negative. To solve this problem, we impose a stronger bound for the Chern characters of $\nab$ tilt-stable objects. In particular, we get a new stronger bound for slope stable sheaves on Fano $3$-folds of Picard number one:
\begin{theorem}[Corollary \ref{cor:noslstabinRd}]
Let $X$ be a Fano $3$-fold of Picard number one, then for any
$\mu_H$-slope stable coherent sheaf $F$ on $X$, the pair $\widetilde v_H(F)$, which is $\left(1,\chy(F),\che(F)\right)$, is not inside $R_{\frac{3}{2rd}}$.
\label{thm:bndforstabshfinintro}
\end{theorem}
The region $R_{\frac{3}{2rd}}$ is defined in Definition \ref{def:rm}.\\

\textbf{Organization:} In Section \ref{sec1.1}, we collect some classification results of smooth Fano $3$-folds of Picard number one. We will actually use only the bound on the index and degree of Fano $3$-folds but not any other details from the classification results. In particular, we compute the explicit Hirzebruch-Riemann-Roch formula. In Section \ref{sec1.2}, we collect some notations for the stability conditions and set up the picture of the $\cccp$-plane to compare slopes of objects. In Section \ref{sec2}, we prove the conjecture for  Fano $3$-folds of index $2$. In Section \ref{sec3}, we give the proof for Theorem \ref{thm:bndforstabshfinintro} and the main theorem in the index $1$ case. 

\textbf{Other threefolds:} Although it is possible to generalize the proof  to all Fano $3$-folds, several problems may appear in the general situation. For the higher Picard number case, Conjecture 5.3 in \cite{BMS}  does not imply the existence of an open subset of stability conditions. One needs to treat the case when the polarization and the B-field are not proportional. Even in the case that $\omega$ and $B$ are both proportional to a divisor $H$, the second Todd class may not be parallel to $H^2$, therefore the term with $\ch^{\ob_H}_1(E)$ may be negative and hard to control. Meanwhile, depending on $\ob_H(E)$, there are few candidate sheaves $F$ such that $\chi(F,E)\leq 0$. These facts will make the computation very complicated.

For a more interesting case, the Calabi-Yau threefold of Picard number one, one of the key point fails. Suppose $0<\ob(E)<1$, by comparing the slope, we only have Hom $(\mathcal O(H),E)=$ Hom $(E,\mathcal O[1])=0$, so $\chi(\mathcal O(H),E)$ and other $\chi(F,E)$ may be positive.

\textbf{Acknowledgments.} The author is grateful to Arend Bayer and Xiaolei Zhao for very useful comments on the proof of Proposition \ref{lemma:bgcone} and their kind help on revising the paper. The author would also like to thank Mingmin Shen, Kenny Wong for helpful conversations. The author is supported by ERC starting grant no. 337039 ``WallXBirGeom''.


\section{Fano threefolds of Picard number one and Notations Recollection}
In this section, we collect some known facts about smooth Fano $3$-folds of Picard number one. For further details, we refer to \cite{Cut,IP,Shen}.

Let $X$ be a $3$-dimensional smooth projective variety
of Picard number one over $\mathbb C$. The Picard group has a unique \emph{ample generator} $H$. The variety $X$ is
\emph{Fano} in the sense that its anti-canonical divisor $-K_X$ is
ample. The
\emph{index} of $X$, denoted by $r$, is the positive integer such
that  $-K_X=rH$. Let $d$ be $H^3$, and we call it the \emph{degree} of $X$. When the index of $X$ is $1$, $d$ is an even number
$2g-2$, where the positive integer $g$ is called the \emph{genus} of
$X$.

\subsection{Classification of Fano threefolds of Picard number one and HRR formula}\label{sec1.1}

The smooth Fano $3$-folds of Picard number one have been completely classified in \cite{IP,Shen}. The detailed list of Fano $3$-folds requires some notations and properties that are not
closely related to the argument in this paper. We only summarize
the range of the index and degree, which plays an important role
in our arguments.
\begin{theorem}[\cite{IP}]
Let $X$ be a smooth Fano $3$-fold of Picard number one and other notations
as above, then the index $r$ is less than or equal to $4$. In addition,
\begin{itemize}
\item if $r=4$, then $X$ is $\pp$;
\item if $r=3$, then $X$ is a quadric hypersurface in $\mathbf P^4$;
\item if $r=2$, then the degree $d$ is a positive integer less than or equal to $5$;
\item if $r=1$, then the degree $d$ is a positive even integer less than or
equal to $22$, and $d$ cannot be $20$.
\end{itemize}
\label{thm:pic1fano3fold}
\end{theorem}

Let $E$ and $F$ be objects in D$^b(X)$, the bounded
derived category of coherent sheaves on $X$. We write ch$_i(E)$ for the $i$-th Chern character of $E$, $i=0,1,2,3$. The Euler character $\chi(E,F)$ is defined to be the alternating sum
\[\sum_{i\in \mathbb Z} (-1)^i \mathrm{ hom} (E,F[i])\]
as usual, and $\chi(E):=\chi(\mathcal O,E)$. The Hirzebruch-Riemann-Roch formula is crucial in the proof of Theorem
\ref{thm:mainthmfanostab}.
\begin{lemma} [Hirzebruch-Riemann-Roch formula]
Let $X$ be a smooth Fano $3$-fold of Picard number one, index
$r$ and degree $d$. The Hirzebruch-Riemann-Roch formula is given as follows:
\begin{itemize}
\item $r=4$: $X$ $=$ $\pp$, $H$ $=$ $[\mathcal O(1)]$.
\[\chi(E)=\ch_3(E)+2H\,\ch_2(E)+\frac{11}{6}H^2\ch_1(E)+H^3\ch_0(E).\]
\item $r=3$: $X$ $=$ $Q$ $\subset$ $\mathbf P^4$, $H$ $=$ $[\mathcal O_{\mathbf P^4}(1)|_Q]$.
\[\chi(E)=\ch_3(E)+\frac{3}{2}H\,\ch_2(E)+\frac{13}{12}H^2\ch_1(E)+\frac{1}{2}H^3\ch_0(E).\]
\item $r=2$: $H=-\frac{1}{2}K_X$.
\[\chi(E)=\ch_3(E)+H\,\ch_2(E)+(\frac{1}{3}+\frac{1}{d})H^2\ch_1(E)+\frac{1}{d}H^3 \ch_0(E).\]
\item $r=1$: $H=-K_X$.
\[\chi(E)=\ch_3(E)+\frac{1}{2}H\,\ch_2(E)+(\frac{1}{12}+\frac{2}{d})H^2\ch_1(E)+\frac{1}{d}H^3\ch_0(E).\]
\end{itemize}
\label{lemma:hrr}
\end{lemma}
\begin{proof}
By the Hirzebruch-Riemann-Roch theorem, $\chi(E)$ can be written as
$\ch_3(E)+ \frac{1}{2}\ch_1(\mathcal O(-K_X))\ch_2(E)+a_2H^2\ch_1(E)+a_3H^3\ch_0(E)$
for some coefficients $a_2$ and $a_3$. By the Kodaira vanishing theorem, $\HH^i(X,\mathcal O)=$ $\HH^i(X,\mathcal O(-K_X)\otimes\mathcal O(K_X))=0$, when $i\geq 1$. Therefore,

\[\chi(\mathcal O) = \mathrm{hom}(\mathcal O,\mathcal O) = 1.\]
This implies $1=a_3H^3\ch_0(\mathcal O)$ and $a_3=\frac{1}{d}$. 

In order to compute $a_2$ when the index is greater than $1$, by the Kodaira vanishing theorem again, we have

\[\chi(\mathcal O(-H))=\chi\left(\mathcal O(K_X)\otimes \mathcal O\big((r-1)H\big)\right) = \hom (\mathcal O,\mathcal O(-H)) =
0.\]

Substitute this into the HRR formula, for instance, when $r=2$:

\[0 = -\frac{H^3}{6}+H\cdot\frac{H^2}{2} -a_2H^2\cdot (-H)+\frac{1}{d}H^3,\]
we get $a_2$ $=$ $\frac{1}{3}+\frac{1}{d}$. The \textbf P$^3$ and
quadric hypersurface cases can be solved in the same way.

When $r=1$, by Serre duality, $\chi(\mathcal O(-H))$ $=$
$-\chi(\mathcal O)$ $=$ $-1$. Substitute this value back to the HRR
formula, we get $a_2$   $=\frac{1}{12}+\frac{2}{d}$.
\end{proof}

\subsection{The space of tilt slope functions}\label{sec1.2}

In this section we recollect some notations of stability conditions
and useful lemmas.

Let $E$ be an object in D$^b(X)$. We write ch$^\beta_i(E)$ for the $\beta H$-twisted $i$-th Chern character ch$_i(E(-\beta H))$. The $H$-slope $\mu_H(E)$ is defined as $\frac{H^2\ch_1(E)}{H^3\ch_0(E)}$. The notations  of tilt heart Coh$_{\beta}(X)$, the reduced central charge $\overline Z_{\alpha,\beta}$, tilt slope function $\nab$, tilt-stability and $\ob$-stability are defined in the introduction. We denote the Grothendieck group of D$^b(X)$ as K$(X)$. The $\mathbb R$-linear span of the Grothendieck group is K$_{\mathbb R}(X):= $ K$(X)\otimes_{\mathbb Z}\mathbb R$.\\

The character map $\overline{v}_H$ maps objects in D$^b(X)$ to their degrees, which is a vector in $\mathbb R^4$:
\[\overline{v}_H(E):=\left(H^3\ch_0(E),H^2\ch_1(E),H\, \ch_2(E),\ch_3(E)\right).\]
The map $\overline {v}_H$ factors through  K$(X)$.
The reduced character map $\widetilde{v}_H$ maps some objects in D$^b(X)$ to the real projective plane P($\mathbb R ^3$):
\[\widetilde{v}_H(E):=\left[H^3\ch_0(E),H^2\ch_1(E),H\, \ch_2(E)\right]=\left(1, \frac{H^2\ch_1(E)}{H^3\ch_0(E)},\frac{H\, \ch_2(E)}{H^3\ch_0(E)}\right).\]
The map $\widetilde{v}_H$ is well-defined only when ($H^3\ch_0(E)$, $H^2\ch_1(E)$, $H\, \ch_2(E)$) is non-zero.
The last equality makes sense when $\widetilde{v}_H(E)$ is not on the line at infinity, in other words, when $\ch_0(E)\neq 0$. We call this image space P($\mathbb R ^3$) the \emph{projective $\cccp$-plane} and  P($\mathbb R ^3)\setminus\{\ch_0=0\}$ the \emph{$\cccp$-plane}.

By $H$-\emph{discriminant}, we denote
\[\overline \Delta_H(E) := (H^2\ch_1(E))^2-2H^3\ch_0(E) H\, \ch_2(E);\]
note that $\db_H$ factors via $\overline v_H$. The \textit{reduced discriminant} is given by
\[\widetilde \Delta_H(E):= \left(\frac{H^2\ch_1(E)}{H^3\ch_0(E)}\right)^2-2\frac{ H\, \ch_2(E)}{H^3\ch_0(E)}.\]
It is well-defined when $H^3\ch_0(E)\neq 0$. It factors via $ \tv_H$ and is defined on the $\cccp$-plane.  We define $\widetilde \Delta_m$ as the parabola $\widetilde \Delta_H = m$ in the $\cccp$-plane, and write  $\dt_{<0}$ or $\db_{< 0}$ for the area $\dt_H <0$.\\

Given a tilt slope function $\nab$, the kernel of its reduced central charge $\overline{Z}_{\alpha,\beta}$ in K$_\mathbb R(X)$ is independent of $\ch_3$. The kernel is explicitly written as
\[\left\{w\in \text{K}_\mathbb R(X) |\; H^2 \ch_1(w)=\beta H^3\ch_0(w);\;\;\; H\, \ch_2(w)=\frac{\alpha^2+\beta^2}{2}H^3\ch_0(w)\right\}.\]
As mentioned in the introduction, we visualize  $\nab$ via the kernel of its reduced central charge in the $\cccp$-plane, or explicitly, the point with coordinate $\left(1,\beta,\frac{\alpha^2+\beta^2}{2}\right)$. Now if $P$ is a point in $\dt_{<0}$ in the $\cccp$-plane, then there is a unique tilt slope function $\nab$ whose kernel is $P$, and we denote it by $\nu_P$. Let $E$ be an object in Coh$_\beta(X)$ such that $\db_H(E)\geq 0$; the kernel of the formal reduced central charge $\overline{Z}_{0,\ob(E)}$ in the $\cccp$-plane is the point on the curve $\dt_0$ such that its tangent line of $\dt_0$ passes through $\widetilde{v}_H(E)$. Though there are two candidate tangent points, the point is uniquely determined as follows: when ch$_0(E)\geq 0$, $\ob(E)$ is  less than $\beta$;  when ch$_0(E)< 0$, $\ob(E)$ is greater than or equal to $\beta$.

When we use the $\cccp$-plane, we always assume that  $\frac{H^2\ch_1}{H^3\ch_0}$ is the horizontal axis and $\frac{H\, \ch_2}{H^3\ch_0}$ is the vertical axis. Let $E$ and $F$ be two objects in $D^b(X)$ with non-zero $\widetilde v_H$. Let $\nab=\nu_P$ be a tilt slope function, we write $L_{EF}$ ($l_{EF}$) for the straight line (segment) in the projective $\cccp$-plane across $\widetilde v_H(E)$ and $\widetilde v_H(F)$; $L_{E\nab}$ or $L_{EP}$ for the line passing through $\widetilde v_H(E)$ and $P$. We write $L_{E\pm}$ for the line passing through $\tv_H(E)$ and $(0,0,1)$. We write $L_{E+}$ for the ray in the $\cccp$-plane from $\widetilde v_H(E)$ to $(0,0,1)$ (upward direction) and $L_{E-}$ for the ray from $E$ to $(0,0,-1)$ (downward direction).

\begin{center}

\tikzset{%
    add/.style args={#1 and #2}{
        to path={%
 ($(\tikztostart)!-#1!(\tikztotarget)$)--($(\tikztotarget)!-#2!(\tikztostart)$)%
  \tikztonodes},add/.default={.2 and .2}}
}

\begin{tikzpicture}[domain=2:1]
\newcommand\XA{0.1}
\newcommand\obe{-0.3}

\coordinate (E) at (1.2,0.5);
\node  at (E) {$\bullet$};
\node [below] at (E) {$E$};

\coordinate (F) at (3,-1);
\node [left] at (F) {$F$};
\node at (F) {$\bullet$};

\coordinate (P) at (-0.5,1.85);
\node [above] at (P) {$\nu_P$};
\node at (P) {$\bullet$};

\coordinate (B) at (\obe,\obe*\obe/2);
\node at (B) {$\bullet$};
\node[below] at (B) {$\overline{\beta}(F)$};

\draw [add =0 and 1,opacity =1] (F) to (B) node[above, opacity =1] {tangent line};

\draw [add =0.5 and 2.5,opacity =1] (F) to (E);




\draw[->] [opacity=\XA] (-4,0) -- (,0) node[above right] {$w$}-- (4,0) node[above right, opacity =1] {$\frac{H^2\ch_1}{H^3\ch_0}$};

\draw[->][opacity=\XA] (0,-2.5)-- (0,0) node [above right] {O} --  (0,5) node[right, opacity=1] {$\frac{H\,\ch_2}{H^3\ch_0}$};

\draw [thick](-3,4.5) parabola bend (0,0) (3,4.5) node [above, opacity =1] {$\widetilde{\Delta}=0$};

\end{tikzpicture} 

Cartoon: $L_{EF}$ and $\ob(F)$.

\end{center}

The following useful lemmas show the convenience for representing $\nab$ by the kernel of its reduced central charge in the $\cccp$-plane.  Their proofs follow immediately from the definitions.

\begin{lemma}
Let $\nu_P$ be a tilt slope function with reduced central charge $\overline{Z}_P$. Let $E$ and $F$ be two objects in $\mathrm{Coh}_P(X)$ such that their Chern characters $v$ and $w$ are not $0$,  then
\begin{center}$\overline{Z}_P(E)$ and  $\overline{Z}_P(F)$ are on the same ray
\end{center}
if and only if $v$, $w$ and $P$ are collinear in the projective $\cccp$-plane.
\label{lemma:paraandspanplane}
\end{lemma}
\begin{proof}
$\overline{Z}_P(v)$ and $ \overline{Z}_P(w)$ are on the same ray if and only if $\overline {Z}_P(av-bw)$  $=$ $0$ for some
$a,b\in\mathbb R_{>0}$. This implies that $v$, $w$ and Ker$\overline{Z}_P$ are collinear in the $\cccp$-plane.
\end{proof}

\begin{defn}
Let $l^+_{PE}$ be the ray from $P$ along $L_{PE}$ to the direction that $\frac{\ch_1}{\ch_0}$ tends to positive infinity. See the figure below.
\end{defn}

\begin{lemma}
Let $P$ be a point in $\dt_{< 0}$ in the $\cccp$-plane, $E$ and
$F$ be two objects in $\mathrm{Coh}_{\beta}$. The inequality
\[\nu_P(E)>\nu_P(F)\] holds if and only if the ray $l^+_{PE}$ is
above $l^+_{PF}$. \label{lemma:slopecompare}
\end{lemma}


\begin{center}

\tikzset{%
    add/.style args={#1 and #2}{
        to path={%
 ($(\tikztostart)!-#1!(\tikztotarget)$)--($(\tikztotarget)!-#2!(\tikztostart)$)%
  \tikztonodes},add/.default={.2 and .2}}
}

\begin{tikzpicture}[domain=2:1]
\newcommand\XA{0.1}
\newcommand\obe{-0.3}

\coordinate (E) at (-3,-2);
\node  at (E) {$\bullet$};
\node [below] at (E) {$E$};

\coordinate (F) at (3,-1);
\node [left] at (F) {$F$};
\node at (F) {$\bullet$};

\coordinate (P) at (0.5,1.5);
\node [above] at (P) {$\nu_P$};
\node at (P) {$\bullet$};

\draw [dashed] (P) -- (F);
\node at (2,0.5) {$l^+_{PF}$};

\draw  (P) -- (E);
\draw [add =0 and 1,dashed] (E) to (P) node [right]{$l^+_{PE}$};

\draw[->][opacity=0.5] (0.5,-2)-- (P)  --  (0.5,4) node [right, opacity =0.5] {$\Im \overline Z_P$};
\draw[->][opacity=0.5] (P)  --  (4,1.5) node [right, opacity =0.5] {$\Re \overline Z_P$};




\draw[->] [opacity=\XA] (-4,0) -- (,0) node[above right] {$w$}-- (4,0) node[above right, opacity =1] {$\frac{H^2\ch_1}{H^3\ch_0}$};

\draw[->][opacity=\XA] (0,-2.5)-- (0,0) node [above right] {O} --  (0,5) node[right, opacity=1] {$\frac{H\,\ch_2}{H^3\ch_0}$};

\draw [thick](-3,4.5) parabola bend (0,0) (3,4.5) node [above, opacity =1] {$\widetilde{\Delta}=0$};

\end{tikzpicture} 

Figure: compare slopes on the $\cccp$-plane

\end{center}


\begin{lemma}
Let $\nu$ be a tilt slope function and $F$ be a
$\nu$ tilt-stable object, then for any tilt slope function
$\tau$ on the line/wall $L_{F\nu}\bigcap \dt_{<0}$, $F$ is also
$\tau$ tilt-stable.
\label{lemma:stabonthewholewall}
\end{lemma}

The main theorem of this paper reads as follows.

\begin{theorem}
Conjecture 5.3 in \cite{BMS} holds for smooth Fano $3$-fold of
Picard number one: Let $X$ be a smooth Fano $3$-fold of Picard number one, and $E\in D^b(X)$ be a $\ob$-stable object, then
\[\czs\leq 0.\]
\label{thm:mainthmfanostab}
\end{theorem}

\section{Index $2$ case}\label{sec2}
In this section, we prove Theorem
\ref{thm:mainthmfanostab} when the $3$-fold has index $2$.

\begin{proof}
Let $E$ be a $\ob$-stable object as in the
Conjecture 5.3 in \cite{BMS}, we write $\ob(E)$ as $\ob$ for short. By Lemma 5.6 in \cite{BMS}, we may assume that $\tv_H(E)$ is not on $\dt_0$. Replacing $E$ by $E(mH)$ for suitable integer $m$, we may assume $0\leq \ob<
1$.  By Proposition 3.10 in \cite{BMS}, the coherent sheaf $\mathcal O(H)$ is also $\nu_{\alpha,\ob}$ tilt-stable for any $\alpha>0$. Since the line
segment $l_{\nu_{0,\ob}\mathcal O(H)}$ is above the
ray $l^+_{\nu_{0,\ob}E}$ which is on the tangent line, by Lemma \ref{lemma:slopecompare}, the tilt slope
$\nu_{\epsilon,\overline\beta}(\mathcal O(H))>\nu_{\epsilon,\overline\beta}(E)$  for any positive $\epsilon$ small enough. We therefore have
\[\Hom(\mathcal O(H),E)=0.\]
The object  $\mathcal O(-H)[1]$ is also
$\nu_{\epsilon,\ob}$ tilt-stable by Proposition 3.10 in \cite{BMS}. By Lemma \ref{lemma:slopecompare},
$\nu_{\epsilon,\ob}(\mathcal O(-H)[1])$ $<$
$\nu_{\epsilon,\ob}(E)$. It follows that

\[\Hom(E,\mathcal O(-H)[1])=0.\]

When $n$ is an even number, Hom($\mathcal O(H), E[n])$ can be
non-zero only when $n$ is $0$ or $2$  because $X$ has dimension $3$, and only H$^0(E)$ and H$^{-1}(E)$ might be non-zero. By Serre duality, we have
\begin{align*}
\chi(\mathcal O(H),E)& \leq \hom(\mathcal O(H),E) +\mathrm{ext}^2(O(H),E)\\
& = \hom(\mathcal O(H),E) +\hom(E,\mathcal O(-H)[1])= 0. \end{align*}
 By the HRR formula in
Lemma \ref{lemma:hrr}, we have
\begin{align*}
0 &\geq \chi(\mathcal O(H),E) = \ch_3(E)+\left(-\frac{1}{6}+\frac{1}{d}\right)\ch_1\\
& =\czs+ \ob H\cze+
\left(\frac{\overline\beta^2}{2}+\frac{1}{d}-\frac{1}{6}\right)H^2\czy
+\left(\frac{\overline\beta^3}{6}+\overline\beta(\frac{1}{d}-\frac{1}{6})\right)H^3\czl\\
& = \czs+
\left(\frac{\overline\beta^2}{2}+\frac{1}{d}-\frac{1}{6}\right)H^2\czy
+\left(\frac{\overline\beta^3}{6}+\overline\beta(\frac{1}{d}-\frac{1}{6})\right)H^3\czl.
\end{align*}
By the definition of $\ob(E)$, we have $\cze=0$ and get the last equality. By the classification result Theorem \ref{thm:pic1fano3fold}, the
degree $d\leq 5$. The coefficients of
$H^2\czy$ and $H^3\czl$ are
greater than $0$. By the definition of $\overline \beta(E)$, when $E$ is not on $\dt_0$, $E$ is in Coh$_{\ob} (X)$, this implies $H^2\czy\geq 0$. When $H^3
\czl$ $\geq$ $0$ or $\overline \beta=0$, we have
\[\left(\frac{\overline\beta^2}{2}+\frac{1}{d}-\frac{1}{6}\right)H^2\czy
+\left(\frac{\overline\beta^3}{6}+\overline\beta(\frac{1}{d}-\frac{1}{6})\right)H^3\czl\geq 0.\]
Hence, we have $\czs\leq 0$ in these two cases.\\

When $0<\overline\beta <1$, due to a similar argument, we have
\[\text{Hom}(\mathcal O(2H),E) = \text{Hom}(E,\mathcal O[1]) = 0.\]
Again by Lemma \ref{lemma:hrr} and $\cze=0$, we have
\begin{align*}
0 &\geq \chi(\mathcal O(2H),E)\\
& = \czs+ (\ob-1)H\cze \\
& + \left(\frac{\overline\beta^2}{2}-\overline \beta
+\frac{1}{d}+\frac{1}{3}\right)H^2\czy+\left(\frac{\overline\beta^3}{6}-\frac{\overline\beta^2}{2}+(\frac{1}{d}+\frac{1}{3})\overline
\beta -\frac{1}{d}\right)H^3\czl.
\end{align*}
As $d$ $\leq$ $5$, the coefficient of $H^2\czy$
is greater than $\frac{\overline\beta^2}{2}-\overline \beta
+\frac{1}{2}\geq 0$. The derivative of
$\frac{\overline\beta^3}{6}-\frac{\overline\beta^2}{2}+\left(\frac{1}{d}+\frac{1}{3}\right)\overline
\beta -\frac{1}{d}$ is $\frac{\overline\beta^2}{2}-\overline \beta
+\frac{1}{d}+\frac{1}{3}$, and is nonnegative. When $\ob=1$, the function
$\frac{\overline\beta^3}{6}-\frac{\overline\beta^2}{2}+\left(\frac{1}{d}+\frac{1}{3}\right)\overline
\beta -\frac{1}{d}=0$. Hence $\frac{\overline\beta^3}{6}-\frac{\overline\beta^2}{2}+\left(\frac{1}{d}+\frac{1}{3}\right)\overline
\beta -\frac{1}{d}<0$,  when $\overline
\beta <1$. When $\czl$ $<$ $0$, we have
\[\left(\frac{\overline\beta^2}{2}-\overline \beta
+\frac{1}{d}+\frac{1}{3}\right)H^2\czy+\left(\frac{\overline\beta^3}{6}-\frac{\overline\beta^2}{2}+(\frac{1}{d}+\frac{1}{3})\overline
\beta -\frac{1}{d}\right)H^3\czl\geq 0.\]
Therefore, $\czs\leq 0$ in all cases.
\end{proof}
\begin{rem}
The case of $\pp$ and quadric $3$-fold in $\mathbf P^4$ has
already been set up in \cite{BMT14,MacP3,Sch} .
The same argument  we used for the index $2$ case also works, the
computation is simpler since there is no term with $d$. Details
are left to the reader.
\end{rem}

\section{Index $1$ case}\label{sec3}

The index $1$ case is more complicated. Because we do not have
the inequality $\chi(\mathcal O(2H),E)\leq 0$, the same argument as that in
the index $2$ case only gives partial results. To solve this problem, we need a stronger bound for $\left(\chy, \che\right)$ of a tilt-stable object.

\begin{defn}
Let $m\geq 0$ be a real number. The open region $R_m$  is defined on the $\cccp$-plane as the set of points above   the curve $\widetilde \Delta_m$, and above tangent lines of the curve
$\dt_0$ at $\tv_H(\mathcal O(k))$   for all $k\in \mathbb Z$.
\label{def:rm}
\end{defn}


\begin{center}

\tikzset{%
    add/.style args={#1 and #2}{
        to path={%
 ($(\tikztostart)!-#1!(\tikztotarget)$)--($(\tikztotarget)!-#2!(\tikztostart)$)%
  \tikztonodes},add/.default={.2 and .2}}
}

\begin{tikzpicture}[domain=2:1]
\newcommand\XA{0.1}
\newcommand\obe{-0.3}
\newcommand \np{0.1}
\newcommand \nb{0.4472135955}

\tikzset{%
    tgt/.style args={#1 and #2}{
        to path={%
 ($(#1+\nb,#1*#1/2+\nb*\nb/2+#1*\nb-\np)$)--($(#1-\nb,#1*#1/2+\nb*\nb/2-#1*\nb-\np)$)%
  \tikztonodes},tgt/.default={.2}}
}

\tikzset{%
    tgtr/.style args={#1 and #2}{
        to path={%
 ($(#1,#1*#1/2)$)--($(#1-\nb,#1*#1/2+\nb*\nb/2-#1*\nb-\np)$)%
  \tikztonodes},tgtr/.default={.2}}
}

\tikzset{%
    tgtl/.style args={#1 and #2}{
        to path={%
 ($(#1+\nb,#1*#1/2+\nb*\nb/2+#1*\nb-\np)$)--($(#1,#1*#1/2)$)%
  \tikztonodes},tgtl/.default={.2}}
}

\draw [tgt= -2 and 1] (0,0) to (0,0);
\draw [tgt= -1 and 1] (0,0) to (0,0);
\draw [tgt= 0 and 1] (0,0) to (0,0);
\draw [tgt= 1 and 1] (0,0) to (0,0);
\draw [tgt= 2 and 1] (0,0) to (0,0);

\draw [tgtl= -3 and 1] (0,0) to (0,0);
\draw [tgtr= 3 and 1] (0,0) to (0,0);




\draw[->] [opacity=\XA] (-4,0) -- (,0) node[above right] {$w$}-- (4,0) node[above right, opacity =1] {$\frac{H^2\ch_1}{H^3\ch_0}$};

\draw[->][opacity=\XA] (0,-1)-- (0,0) node [above right] {O} --  (0,5) node[right, opacity=1] {$\frac{H\,\ch_2}{H^3\ch_0}$};

\draw [opacity =0.5](-3,4.5) parabola bend (0,0) (3,4.5) node [above, opacity =1] {$\widetilde{\Delta}=0$};
\draw [opacity =0.5](-3,4.5-\np) parabola bend (0,0-\np) (3,4.5-\np) node [right, opacity =0.8] {$\widetilde{\Delta}=\np$};

\draw (1,0.5) node[below right] {$\mathcal O(H)$};

\draw (-1,0.5) node[below left] {$\mathcal O(-H)$};

\draw (0,0) node[below right] {$\mathcal O$};

\draw (2,2) node[below right] {$\mathcal O(2H)$};
\draw (-3,4.5) node[below left] {$\mathcal O(-3H)$};

\draw (-2,2) node[below left] {$\mathcal O(-2H)$};

\foreach \k in {-2,-1,-0,1,2,3}
\pgfmathsetmacro \y {\k-1+\nb}
\pgfmathsetmacro\z{\k-0.5}
\pgfmathsetmacro \w{\k-\nb}
\draw[color=blue] plot [domain=\y:\w] (\x,{\x*\x/2-\np});

\end{tikzpicture} 

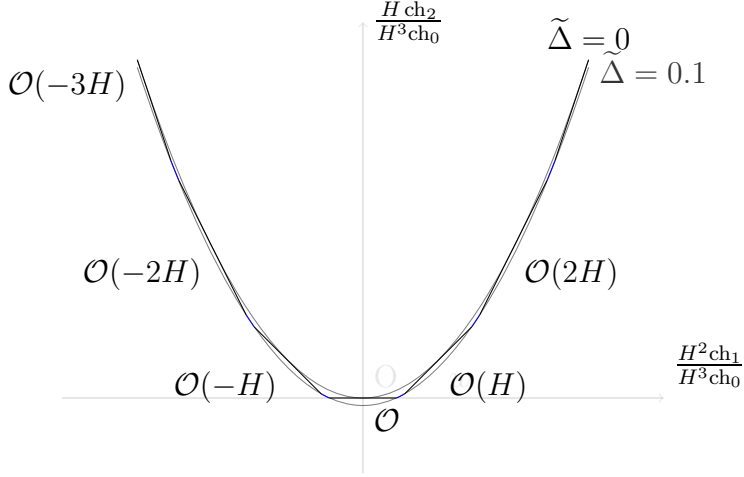
Figure: $R_{\frac{1}{10}}$ is the region above the broken curve between $\dt_{0\sim 0.1}$.
\end{center}


The following proposition provides a bound
for Chern characters $\tv_H(E)$ of a $\nab$ tilt-stable
object $E$ in the Fano $3$-fold case. It is slightly stronger than
the bound given by the Bogomolov inequality  $\db_H\geq 0$ for general varieties. 

Recall that $r$ is the index and $d$ is the degree of the Fano $3$-fold.

\begin{prop}
Let $\nab$ be a tilt slope function and $E$ be a $\nab$ tilt-stable object, then $\widetilde v_H(E)$ in the $\cccp$-plane is not in
$R_{\frac{3}{2rd}}$. If $\widetilde v_H(E)$ is on the boundary of $R_{\frac{3}{2rd}}$, then $\ch_0(E)$ is $1$ or $2$.\label{lemma:bgcone}
\end{prop}

The proof contains two steps. We first show that such kind of object $E$ or $E[1]$ must be tilt-stable for every $\nab$ (we actually only show this for certain $F$ with the minimum $\db_H$). Then by comparing the slopes and Serre duality, Ext$^2(E,E)=0$. This implies $\chi(E,E)\leq 1$. On the other hand, the bound of $\dt_{\frac{3}{2rd}}$ implies that $\chi(E,E)>0$. Therefore, $\chi(E,E)=1$, and this leads to contradiction to the HRR formula.

\begin{proof}
Consider all objects $F$ with $\tv_H(F)\in R_{\frac{3}{2rd}}$ and for which there are $\alpha>0$, $\beta$ such that $F$ is $\nab$ tilt-stable. Within this set, let $E$ be such that $\db_H(E)$ is minimal.

\begin{lemma}
The above object $E$ does not become strictly tilt-semistable unless on the vertical wall $L_{E+}$.
\label{lemma:Estabeverywhere}
\end{lemma}

\begin{proof}[Proof of Lemma \ref{lemma:Estabeverywhere}]

Suppose $E$ becomes strictly semistable on a wall
$L_{\nu E}$ other than the vertical wall, then we get the filtration
for $E$:
 \[0=E_0\subset E_1\subset\dots\subset E_m=E\]
  with
$\nu$ tilt-stable factors $F_i=E_i/E_{i-1}$.  In the $\cccp$-plane, the reduced characters $\tv_H(F_i)$ are all on the line $L_{\nu E}$. Consider the $\mathbb R$-linear plane spanned by $\tv_H(F_i)$'s and $\tv_H(E)$ in $\mathbb R^3$ (the `affine' image of $K_{\mathbb R}(X)$ under $\tv_H$), the set $\db_H=0$ consists of two lines and cuts the plane into four
components. All points $\tv_H(F_i)$ are in the same component of $\db_{>0}$  as $\tv_H(E)$ (some can be
on the boundary). As a vector, $\tv_H(E)$ is the sum of all
$\tv_H(F_i)$, the points $\tv_H(F_i)$'s are not on
the same side to the ray of $\tv_H(E)$. There exists an $F_i$ such that the line segment $l_{EF_i}$ is contained in $R_{\frac{3}{2rd}}$. Since
$L_{\nu E}$ is not the vertical wall, by Corollary 3.10 in \cite{BMS},
$\overline \Delta_H(F_i)$ is strictly less than $\overline
\Delta_H(E)$. But this contradicts the minimum property of $\db_H (E)$. We conclude that $E$ is $\nu$ tilt-stable for any
$\nu$ on the same side of $\nab$ to the vertical wall $L_{E+}$.
\end{proof}

Back to the proof of Proposition \ref{lemma:bgcone}, on the  vertical wall $L_{E+}$, we consider the filtration for $E[1]$ (or $E$ when $H^3\ch_0(E)<0$). If $E[1]$ (or $E$) is not stable, then due to argument in the lemma, it has a tilt-stable factor $E'$ such that $\tv_H(E')\in R_{\frac{3}{2rd}}$. By the minimum assumption on $\db_H$ and Corollary 3.10 in \cite{BMS}, $E'$ and $E[1]$ have the same $\db_H$, and $\tv_H(E')=\tv_H(E[1])$. Since the stability is an open condition, $E'$ is $\nu$ tilt-stable for $\nu$ on a neighborhood to the right of $L_{E'+}$, and $E'[-1]$ is $\nu$ tilt-stable for $\nu$ on a neighborhood to the left of $L_{E'+}$. Since $E'$ and $E'[-1]$ also have the minimum discriminant, by Lemma \ref{lemma:Estabeverywhere}, either $E'$ or $E'[-1]$ is $\nab$ tilt-stable for any $\alpha>0$.\\

We may assume that $E=E'[-1]$ and $H^3\ch_0(E)> 0$. Since $\widetilde v_H(E)$ is above the tangent line of the curve
$\dt_0$ at
$\tv_H(\mathcal O)$ and $\tv_H\left(\mathcal O(rH)\right)$  , the line segment $l_{EE(-rH)}$
intersects the cone $\widetilde \Delta_{<0}$, and there is a
tilt slope function $\nab$ with kernel below
$l_{EE(-rH)}$. By Lemma \ref{lemma:slopecompare}, we have $\nu_{\alpha,\beta}(E(-rH)[1])$ $<$
$\nu_{\alpha,\beta}(E)$. Since both of them are $\nab$ tilt-stable, we have
\[\Hom (E, E(-rH)[1]) = 0.\]
By Serre duality, we have
\[\chi(E,E) \leq \hom(E,E) + \ext^2(E,E) = 1+
\hom(E,E(-rH)[1])= 1. \]
 On the other hand, by the HRR formula,
\begin{equation*}
\chi(E,E) = -\frac{r}{2d}\overline
\Delta_H(E)+\frac{1}{d}H^3\ch_0(E)\ch_0(E)>0.
\tag{$\blacksquare$}\label{eq0}
\end{equation*}
 The last inequality is
because of \[\widetilde \Delta_H(E)< \frac{3}{2rd} < \frac{2}{rd}\text{ and  }\db_H(E) = \dt_H(E)\left(H^3\ch_0(E)\right)^2.\]
Hence, $\chi(E,E) = 1$. Substitute this into (\ref{eq0}), we have the identity:
\[\widetilde \Delta_H(E) =
\big(1-\frac{1}{\mathrm{rk}(E)^2}\big)\frac{2}{dr}\geq
\big(1-\frac{1}{4}\big)\frac{2}{dr} = \frac{3}{2dr}.\]

Therefore, $\widetilde v_H(E)$ is not in $\widetilde \Delta_{<\frac{3}{2dr}}$
and is on the boundary of $R_{\frac{3}{2rd}}$ only when rk$(E)$ is $1$ or $2$.

\end{proof}
\begin{cor}
Let $X$ be a Fano $3$-fold of Picard number one, then for any
$\mu_H$-slope semistable coherent sheaf $F$ on $X$, $\widetilde v_H(F)$ is
not inside $R_{\frac{3}{2rd}}$.
\label{cor:noslstabinRd}
\end{cor}
\begin{rem}
This kind of bound is trivial to prove in the Fano surface case, since in the surface case the $\mathrm{Ext}^2$ vanishing is directly due to Serre duality. In the Fano $3$-fold case, we need the concept of tilt stability condition.
\end{rem}
Now we are ready to prove Theorem \ref{thm:mainthmfanostab} for the index $1$ case.
\begin{proof}[Proof of the index one case]
Let $E$ be a $\ob$-stable object as in the
Theorem \ref{thm:pic1fano3fold}. We first consider the case that $\ob$ is not an integer. For the same reason as that in the index two case, we may assume that $0< \ob(E)<
1$. By Lemma 5.6 in \cite{BMS}, we may also assume that $\tv_H(E)$ is not on the cone
$\dt_0$. As $0<\ob<1$, by the same argument as that in the
index two case,  we still have
\[\Hom(E,\mathcal O[1])=0\text{ and }\Hom(\mathcal O(H),E)=0.\]
By the same argument, we have
\begin{align*}
0  \geq &  \chi(\mathcal O(H),E) \\
 = & \czs  + \left(\ob-\frac{1}{2}\right)\cze +
\left(\frac{\overline\beta^2}{2}-\frac{\overline \beta}{2}
+\frac{2}{d}+\frac{1}{12}\right)H^2\czy\\
   & +\left(\frac{\overline\beta^3}{6}-\frac{\overline\beta^2}{4}+(\frac{2}{d}+\frac{1}{12})\overline
\beta -\frac{1}{d}\right)H^3\czl \\
= & \czs +
\left(\frac{\overline\beta^2}{2}-\frac{\overline \beta}{2}
+\frac{2}{d}+\frac{1}{12}\right)H^2\czy +\left(\frac{\overline\beta^3}{6}-\frac{\overline\beta^2}{4}+(\frac{2}{d}+\frac{1}{12})\overline
\beta -\frac{1}{d}\right)H^3\czl.
\end{align*}
Let the function $f(x)$ be
\[\frac{x^3}{6}-\frac{x^2}{4}+\left(\frac{2}{d}+\frac{1}{12}\right)x -\frac{1}{d},\]
then the coefficient $\frac{\overline\beta^2}{2}-\frac{\overline
\beta}{2} +\frac{2}{d}+\frac{1}{12}$ for
$\czy$ is $f'(\overline \beta)$. As $d\leq
22$, $f'(\overline \beta)$ is greater than $0$ for any real number
$\overline \beta$. The coefficient for $H^3\czs$
is $f(\overline \beta)$, which is monotonically increasing with
respect to $\overline \beta$. As $f(\frac{1}{2})=0$, in either one
of the two subcases
\begin{itemize}
\item $\frac{1}{2}\leq \overline \beta <1$ and $\czl\geq 0$;
\item or $0 <\overline \beta \leq \frac{1}{2}$ and $\czl\leq 0$;
\end{itemize}
we have $\czs\leq 0$.

We next show by contradiction that in the other subcases,
$f'(\ob)H^2\czy+f(\ob)H^3\czl$
$\geq$ $0$. In the case when  $0 < \overline \beta < \frac{1}{2}$
and $\ch^{\overline \beta}_0(E) > 0$, assume that
\begin{equation*}f'(\overline\beta)H^2\czy+f(\overline\beta)H^3\czl<0.
\tag{*}\label{eq:rest<0}\end{equation*} Since $\overline \beta(E)$
is greater than $0$, $\widetilde v_H(E)$ is not below the tangent
line of the curve
$\widetilde \Delta_0$ at $\widetilde v_H(\mathcal O)$. By the assumption (\ref{eq:rest<0}) on $E$,
we have
\[f'(\ob)H^2\ch_1(E)+(-\frac{\overline \beta^3}{3}+\frac{\overline \beta^2}{4}-\frac{1}{d})H^3\ch_0(E)<0.\]
The tangent lines  of $\dt_0$ at $\tilde v_H(\mathcal O)$ and $\tilde v_H(\mathcal O(H))$ in the $\cccp$-plane intersect at $(1,\frac{1}{2},0)$. Substitute $\left( H^3\ch_0(E),H^2\ch_1(E)\right)$ by $\left(1,\frac{1}{2}\right)$, we have
\begin{align*}
\frac{1}{2}f'(\overline \beta)+-\frac{\overline
\beta^3}{3}+\frac{\overline \beta^2}{4}-\frac{1}{d}
 =  -\frac{\overline \beta^3}{3}+\frac{\overline \beta^2}{2}-\frac{\overline \beta}{4} +
 \frac{1}{24}
 >  0,
 \end{align*}
when $0< \overline \beta <\frac{1}{2}$. As $f'(\ob)>0$, the left hand side of (\ref{eq:rest<0}) is not less than $0$ when $\frac{H^2\ch_1(E)}{H^3\ch_0(E)}\geq \frac{1}{2}$.  Therefore, the point $\widetilde v_H(E)$ has $\frac{H^2\ch_1}{H^3\ch_0}$-coordinate less than $\frac{1}{2}$ and is above the tangent line of $\dt_0$ at $\widetilde v_H(\mathcal O(H))$.


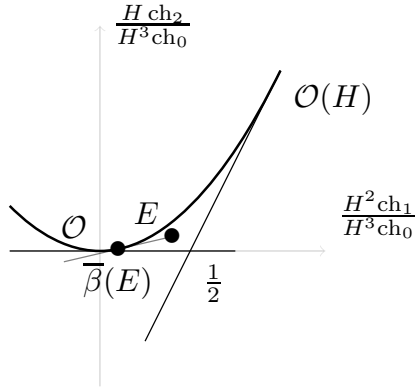
\begin{center}

\scalebox{1.2}{
\tikzset{%
    add/.style args={#1 and #2}{
        to path={%
 ($(\tikztostart)!-#1!(\tikztotarget)$)--($(\tikztotarget)!-#2!(\tikztostart)$)%
  \tikztonodes},add/.default={.2 and .2}}
}

\begin{tikzpicture}[domain=2:1]
\newcommand\XA{0.1}
\newcommand\obe{0.2}

\coordinate (E) at (0.8,0.16);
\node [above left,font=\footnotesize] at (E) {$E$};
\node at (E) {$\bullet$};

\coordinate (B) at (\obe,\obe*\obe/2);
\node at (B) {$\bullet$};
\node[below,font=\footnotesize] at (B) {$\overline{\beta}(E)$};

\draw [add =0 and 1,opacity =0.5] (E) to (B);




\draw[->] [opacity=\XA] (-1,0) -- (,0) node[above right] {}-- (2.5,0) node[above right, opacity =1,font=\footnotesize] {$\frac{H^2\ch_1}{H^3\ch_0}$};

\draw[->][opacity=\XA] (0,-1.5)-- (0,0) node [above left, opacity=1,font=\footnotesize] {$\mathcal O$} --  (0,2.5) node[right, opacity=1,font=\footnotesize] {$\frac{H\,\ch_2}{H^3\ch_0}$};

\draw [thick](-1,0.5) parabola bend (0,0) (2,2) node [above, opacity =1,font=\footnotesize] {};

\draw (-1,0) -- (1.5,0);

\draw (2,2) -- (0.5,-1);

\draw (1,0) node [below right,font=\footnotesize] {$\frac{1}{2}$};
\draw (2,2) node [below right,font=\footnotesize] {$\mathcal O(H)$};

\end{tikzpicture} 
}\\
Cartoon: $\tv_H(E)$ is above the tangent lines at $\tv_H(\mathcal O)$ and $\tv_H(\mathcal O(H))$.
\end{center}


Consider the function $g(x): = \sqrt{\frac{3}{2d}}f'(x)+f(x)$, we have
\[g(0) = \sqrt{\frac{3}{2d}}\Big(
\frac{1}{12}+\frac{2}{d}\Big)-\frac{1}{d} =
\sqrt{\frac{3}{2d}}\Big(\sqrt{\frac{2}{d}}-\frac{1}{2}\sqrt{\frac{1}{3}}\Big)^2>0.\]
The derivative of $g(x)$ is
\begin{align*}
g'(x) & =
\frac{x^2}{2}-\Big(\frac{1}{2}-\sqrt{\frac{3}{2d}}\Big)x+\frac{1}{12}+\frac{2}{d}-\frac{1}{2}\sqrt{\frac{3}{2d}}\\
 & \geq g'\Big(\frac{1}{2}-\sqrt{\frac{3}{2d}}\Big) =
 -\frac{1}{2}\Big(\frac{1}{2}-\sqrt{\frac{3}{2d}}\Big)^2+\frac{1}{12}+\frac{2}{d}-\frac{1}{2}\sqrt{\frac{3}{2d}}\\
 & = \frac{5}{4d} - \frac{1}{24} > 0.
\end{align*}
The last inequality is because of $d$
$\leq$ $22$. Therefore, $g(x)>0$, when $x>0$. Since $f'(x) >0$, this implies
$\Big(\frac{f(x)}{f'(x)}\Big)^2 < \frac{3}{2d}$. Now we may estimate
the reduced discriminant of $E$:
\begin{align*}
\widetilde \Delta_H(E) = \widetilde \Delta^{\ob}_H(E) =
\left(
\frac{H^2\czy}{H^3\czl}\right)^2
\leq \left( \frac{-f(\overline \beta)}{f'(\overline \beta)}\right)^2
<\frac{3}{2d}.
\end{align*}

Therefore, $\widetilde v_H(E)$ is in $R_{\frac{3}{2d}}$, this leads the contradiction to Proposition
\ref{lemma:bgcone}. The rest case when $\frac{1}{2}<\overline\beta<1$ and $\ch^{\overline
\beta}_0(E)<0$ is proved in the same way.\\

We now treat with the case that $\ob(E)$ is an integer. We may assume that $\ob(E)=0$. 

If Hom$(E,\mathcal O[1])= 0$, then $\chi(\mathcal O(H),E)\leq 0$ and by the same argument as above, this will imply $\ch_3(E)\leq 0$. If Hom$(\mathcal O, E)=0$, then $\chi(\mathcal O,E)\leq 0$ and by a similar argument as above, this will imply $\ch_3(E)\leq 0$. Therefore, $\ch_3(E)>0$ implies both Hom$(\mathcal O, E)\neq 0$ and Hom$(E,\mathcal O[1])\neq 0$.

Suppose there exists a $\ob$-stable object $E$ such that $\ob(E)=0$ and  $\ch_3(E)>0$, among all such objects, we assume that $E$ has the minimum discriminant $\db_H$, in other words, $E$ has the minimum $|H^2\ch_1|$. By taking the derived dual R$\mathcal Hom(E,\mathcal O)[1]$ if necessary, we may assume that  $\ch_0(E)\geq 0$.  There is a negative number $c$ such that for any $c<\beta<0$, $E$ and $\mathcal O$ are both in the heart Coh$_\beta(X)$. By assumption, there exists a non-zero $\iota$ in Hom$(E,\mathcal O[1])$ inducing a non-trivial extension $F$ of $E$ by $\mathcal O$. The extension $F$ does not depend on the choice of $\beta$.

\begin{lemma}
The object $F$ constructed as above is $\ob$-stable.
\label{lemma:Fisobstable}
\end{lemma}
\begin{proof}[Proof of Lemma \ref{lemma:Fisobstable}]
We first show that the object $F$ is in Coh$_0(X)$. For any $H$-slope stable sheaf $G$ such that $\mu_H(G)\leq 0$, Hom$(\mathcal O, G)\neq 0$ only when $G=\mathcal O$. Since $E$ is in Coh$_0(X)$, Hom$(E, G)=0$. But since $\iota$ is non-zero, Hom$(F,\mathcal O)=0$. Apply Hom$(-,G)$ to the exact triangle
\[\mathcal O\rightarrow F\rightarrow E\xrightarrow{+} \tag{$\bigstar$} \label{eq1}\]
we get Hom$(F,G)=0$. Therefore, $\HH^0(F)$ has no quotient object with non-positive $\mu_H$ slope. Since $\HH^{-1}(F)$ is the kernel of $\HH^{-1}(E)\rightarrow \mathcal O$ and $E$ is in Coh$_0(X)$, it has no subobject with positive slope.  By the definition of Coh$_0(X)$, $F$ is in Coh$_0(X)$.

Assume that $F$ is not $\ob$-stable, then there is a filtration of $F$:
\[0=F_0\subset F_1\subset\dots\subset F_m=F\]
in Coh$_0(X)$ such that for any $\alpha$ small enough, each factor  $K_i:=F_{i}/F_{i-1}$ is $\nu_{\alpha,0}$ tilt-stable and $\nu_{\alpha,0}(K_{i+1}) \leq \nu_{\alpha,0}(K_{i})$. As $\mathcal O[1]$ and $F_i$ are in a same heart Coh$_0(X)$,
\[\Hom(F_i,\mathcal O) = \Hom(F_i,(\mathcal O[1])[-1]) = 0.\]
Since $\Hom(F_i,F)\neq 0$ and $F$ is the extension of $E$ by $\mathcal O$, $\Hom(F_i,E)\neq 0$. Since $E$ is $\ob$-stable,
\[\nu_{\alpha,0}(F)\leq \nu_{\alpha,0}(F_i)\leq \nu_{\alpha,0}(E)\]
for $\alpha$ small enough. By Lemma \ref{lemma:slopecompare} and the inequalities above, $\tv_H(F_i)$ is well-defined and has to be on the line segment $l_{EF}$. Therefore, $\ch_2(K_i)=0$ for all $1\leq i\leq n$. By the definition of $\ob$, $\ob(K_i)=0$ for all $1\leq i\leq n$. By Lemma \ref{lemma:stabonthewholewall} and the definition of $\ob$-stable, each $K_i$ is $\ob$-stable.

Since $H^2\ch_1(K_i)\geq 0$ and $\sum_i \ch_1(K_i)=\ch_1(F)$, $0\leq H^2\ch_1(K_i)\leq H^2\ch_1(F)$ for any $1\leq i\leq n$. Since $\sum_i \ch_3(K_i)=\ch_3(F)=\ch_3(E)>0$, there exists $K_i$ with positive $\ch_3$. By the minimum assumption on $|H^2\ch_1|$ of $E$, the only possible case is that $\ch_1(K_l)=\ch_1(E)$ for an $l$ and $\ch_1(K_i)=0$, $\ch_3(K_i)\leq 0$ when $i\neq l$. $K_l$ mush be the first factor $K_1(=F_1)$ since otherwise $\tv_H(F_1)$ is at the origin and is not on the line segment $l_{EF_i}$. Since $\ch_1(F_1)=\ch_1(E)=\ch_1(F)$ and $\mathrm{rk}(F)-\mathrm{rk}(E)=1$, $\tv_H(F_1)$ is either the same as $\tv_H(F)$ or $\tv_H(E)$. If $\tv_H(F_1)=\tv_H(F)$, then either $F$ is $\ob$-stable or $\nu_{\alpha,0}(F_1)\geq \nu_{\alpha,0}(F/F_1)=+\infty$, either case will lead to a contradiction. If $\tv_H(F_1)=\tv_H(E)$, then
since $E$ and $F_1$ are both $\ob$-stable, this can happen only when $F_1\hookrightarrow E\rightarrow C$ in Coh$_0(X)$ and the cokernel $C$ is zero or a sheaf supported on dimension $0$. Since $\ch_3(K_i)\leq 0$ when $i>1$, $\ch_3(F_1)\geq \ch_3(E)$. Therefore $F_1$ must be $E$. Now apply $\Hom(F_1,-)=\Hom(E,-)$ to the vanishing triangle (\ref{eq1}), we get the exact sequence:
\[\Hom(E,\mathcal O)\rightarrow \Hom(E,F)\rightarrow \Hom(E,E)\rightarrow \Hom(E,\mathcal O[1]).\]
Since $\Hom(E,E)=\mathbb C$ and the image of its identity in $\Hom(E,\mathcal O[1])$, which is $\iota$, is nonzero, $E$ is not a subobject of $F$. This leads to the contradiction, and $F$ is $\ob$-stable.

\end{proof}

Back to the proof of the main theorem: The Chern characters of $F$ are $(\ch_0(E)+1$, $\ch_1(E)$, $0$, $\ch_3(E))$. Since $F$ still satisfies the minimum assumption on $|\ch_1(F)|$, by Lemma \ref{lemma:Fisobstable}, we may construct a sequence of $\ob$-stable objects with Chern characters \[\left(\ch_0(E)+n, \ch_1(E), 0, \ch_3(E)\right).\]
When $n\gg 0$, this character is on the boundary of $R_{\frac{3}{2d}}$, this leads the contradiction to Proposition \ref{lemma:bgcone}.
\end{proof}

\bibliographystyle{abbrv}\bibliography{exc}

\end{document}